\documentclass[11pt,a4paper]{amsart}
\usepackage[T1]{fontenc}
\usepackage{lmodern}
\usepackage[a4paper,margin=28mm]{geometry}
\usepackage{amssymb}
\usepackage[colorlinks=true,linkcolor=blue,citecolor=blue,urlcolor=blue]{hyperref}

\newtheorem{maintheorem}{Theorem}

\newtheorem{theorem}{Theorem}[section]
\newtheorem{proposition}[theorem]{Proposition}
\newtheorem{corollary}[theorem]{Corollary}
\newtheorem{problem}[theorem]{Problem}
\theoremstyle{remark}
\newtheorem{remark}[theorem]{Remark}
\newcommand{\N}{\mathbb N}
\newcommand{\K}{\mathbb K}
\newcommand{\cfrak}{\mathfrak c}
\DeclareMathOperator{\dens}{dens}
\DeclareMathOperator{\dist}{dist}

\title[Strongly normal subsequences]{Normal sequences without strongly normal subsequences}
\author{Jerzy K\k{a}kol}
\address{Faculty of Mathematics and Computer Science, Adam Mickiewicz
University, Uniwersytetu Pozna\'nskiego 4, 61-614 Pozna\'n, Poland}
\email{kakol@amu.edu.pl}
\date{}
\subjclass[2020]{Primary 46B26; Secondary 46B10, 03E17, 03E35}
\keywords{Normal sequence, strongly normal sequence, weak-star convergence,
separable quotient, bounding number}
\hypersetup{pdftitle={Normal sequences without strongly normal subsequences},
pdfauthor={Jerzy K\k{a}kol}}

\begin{document}

\begin{abstract}
A sequence $(y_n^*)$ in the unit sphere of $E^*$ is strongly normal if
the vectors $x\in E$ satisfying $\sum_n|y_n^*(x)|<\infty$ form a dense
subspace. For every set $\Gamma$ of cardinality at least the continuum,
we construct a normalized weakly null sequence in $\ell_1(\Gamma)^*$
with no strongly normal subsequence. This answers a question of \'Sliwa
negatively in ZFC. Together with the classical selection argument below
the bounding number $\mathfrak b$, the construction gives positive and
negative bounds for this subsequence property on $\ell_1(\kappa)$.
Its validity on $\ell_1(\omega_1)$ is independent of ZFC. The intermediate
range $\mathfrak b\leqslant\kappa<\mathfrak c$ remains undecided by
these results. We also give a counterexample in a Banach sequence space
that is not isomorphic to any $\ell_1(\Gamma)$.
\end{abstract}

\maketitle

\section{Introduction}

The separable quotient problem asks whether every infinite-dimensional
Banach space has an infinite-dimensional separable quotient. We refer
to Mujica \cite{Mujica} and Ferrando, K\k{a}kol, L\'opez-Pellicer, and
\'Sliwa \cite{FKLS} for surveys. Among the general positive results,
Argyros, Dodos, and Kanellopoulos \cite[Theorem~15]{ADK} proved that every
infinite-dimensional Banach space isomorphic to a dual has such a quotient.
The connection between separable quotients and cardinal characteristics
is discussed in Brech's survey \cite{Brech}.

Throughout, Banach spaces are over $\K\in\{\mathbb R,\mathbb C\}$, and
$S(E^*)$ denotes the unit sphere of $E^*$. A sequence
$y^*=(y_n^*)\subset S(E^*)$ is \emph{normal} if it is weak$^*$-null,
and it is \emph{strongly normal} if
\[
 D(y^*)=\left\{x\in E:\sum_{n=1}^{\infty}|y_n^*(x)|<\infty\right\}
\]
is norm dense in $E$. This is a linear subspace. Strong normality implies
normality: for $d\in D(y^*)$ and $x\in E$,
$\limsup_n|y_n^*(x)|\leqslant\|x-d\|$, because $\|y_n^*\|=1$.
Letting $d$ approximate $x$ proves weak$^*$ convergence to zero.

By the Josefson--Nissenzweig theorem \cite{Josefson,Nissenzweig}, every
infinite-dimensional Banach space has a normal sequence in its dual.
\'Sliwa \cite[Theorem~3]{Sliwa} proved that $E$ has an
infinite-dimensional separable quotient if and only if $E^*$ contains
a strongly normal sequence, and asked whether every normal sequence
has a strongly normal subsequence \cite[p.~388]{Sliwa}; see also
\cite[Problem~7]{FKLS}. A positive answer would settle the separable
quotient problem: first choose a normal sequence by the
Josefson--Nissenzweig theorem, then apply \'Sliwa's characterization to
a strongly normal subsequence. The subsequence assertion holds for
every weakly compactly generated Banach space
\cite[Proposition~4]{Sliwa}. More recently, particular sequences of
finitely supported measures on products of compact spaces were shown
to have strongly normal subsequences in \cite[Theorem~2(5)]{KLS}.

We give a negative answer to the unrestricted subsequence question.
Write $\cfrak=2^{\aleph_0}$ and $[\N]^\infty$ for the family of all
infinite subsets of $\N$.

\begin{maintheorem}\label{thm:large}
For every set $\Gamma$ with $|\Gamma|\geqslant\cfrak$, the dual
$\ell_1(\Gamma)^*$ contains a normalized weakly null sequence with no
strongly normal subsequence. Consequently, \'Sliwa's subsequence
question has a negative answer in ZFC, even for weakly null sequences.
\end{maintheorem}

For an infinite cardinal $\kappa$, let $\mathsf P_\kappa$ denote the
assertion that every normal sequence in $\ell_1(\kappa)^*$ has a
strongly normal subsequence. Combining Theorem~\ref{thm:large} with
the classical selection argument below the bounding number
$\mathfrak b$ gives the following bounds.

\begin{maintheorem}\label{thm:bounds}
For every infinite cardinal $\kappa$,
\[
 \kappa<\mathfrak b\ \Longrightarrow\ \mathsf P_\kappa,
 \qquad
 \kappa\geqslant\cfrak\ \Longrightarrow\ \neg\mathsf P_\kappa.
\]
\end{maintheorem}

Since $\aleph_1\leqslant\mathfrak b\leqslant\cfrak$, these bounds
are compatible. They yield independence for $\mathsf P_{\aleph_1}$,
but do not decide the case $\mathfrak b=\aleph_1<\cfrak$. More
generally, the range $\mathfrak b\leqslant\kappa<\cfrak$ is not
settled by our arguments. Section~\ref{sec:density} states the resulting
questions and an obstruction to reducing the index family in the
construction. Section~\ref{sec:sequence-space} gives a second example.

\section{The construction}

For a sequence $y^*=(y_n^*)\subset E^*$ and
$B=\{b_1<b_2<\cdots\}\in[\N]^\infty$, write
\[
 D_B(y^*)=\left\{x\in E:\sum_{j=1}^{\infty}|y_{b_j}^*(x)|<\infty\right\}.
\]
We first give the construction on $\ell_1([\N]^\infty)$, where
the index $B$ itself supplies a vector witnessing the failure of density.

\begin{proposition}\label{prop:construction}
Let $\Gamma_0=[\N]^\infty$ and $E_0=\ell_1(\Gamma_0)$. There is a
normalized weakly null sequence $y^*=(y_n^*)\subset E_0^*$ such that
\[
 \dist(e_B,D_B(y^*))=1\qquad(B\in\Gamma_0),
\]
where $e_B$ is the canonical unit vector indexed by $B$.
\end{proposition}

\begin{proof}
For $f\in\ell_\infty(\Gamma_0)$, set
$f(x)=\sum_{A\in\Gamma_0}x(A)f(A)$.
This identifies $E_0^*$ with $\ell_\infty(\Gamma_0)$, with
$\|f\|_{E_0^*}=\sup_{A\in\Gamma_0}|f(A)|$. For
$A=\{a_1<a_2<\cdots\}\in\Gamma_0$, define
\[
 y_n^*(A)=
 \begin{cases}
  1/j,&n=a_j,\\
  0,&n\notin A.
 \end{cases}
\]
These bounded coefficient functions define elements of $E_0^*$. Their
values lie in $[0,1]$, and $y_n^*(A)=1$ whenever $\min A=n$.
The supremum norm formula therefore gives $\|y_n^*\|=1$. Over
$\mathbb C$, the same real coefficients define complex-linear
functionals by this pairing.

In fact, the sequence is weakly null, so the example remains valid when
weak$^*$ convergence in the definition of normality is strengthened
to weak convergence. Define $U:\ell_2\to\ell_\infty(\Gamma_0)$ by
\[
 (Uz)(A)=\sum_{j=1}^{\infty}\frac{z_{a_j}}j.
\]
The Cauchy--Schwarz inequality gives absolute convergence and
\[
 |(Uz)(A)|
 \leqslant\left(\sum_{j=1}^{\infty}|z_{a_j}|^2\right)^{1/2}
            \left(\sum_{j=1}^{\infty}\frac1{j^2}\right)^{1/2}
 \leqslant\frac\pi{\sqrt6}\|z\|_2.
\]
Thus $U$ is bounded and linear. If $(u_n)$ is the canonical basis of
$\ell_2$, then $(Uu_n)(A)$ equals $1/j$ when $n=a_j$ and equals
zero otherwise. Hence $Uu_n=y_n^*$. Bounded linear operators are
weak-to-weak continuous, so weak nullity of $(u_n)$ proves the claim.

Let us fix $B=\{b_1<b_2<\cdots\}\in\Gamma_0$ and set
\[
 H_m=\sum_{j=1}^m\frac1j,
 \qquad S_m^*=\sum_{j=1}^m y_{b_j}^*.
\]
We have
\begin{equation}\label{eq:partial-sums}
 \|S_m^*\|=S_m^*(e_B)=H_m.
\end{equation}
Indeed, for $A=\{a_1<a_2<\cdots\}\in\Gamma_0$, write
\[
 A\cap\{b_1,\ldots,b_m\}=\{a_{r_1},\ldots,a_{r_q}\},
 \qquad r_1<\cdots<r_q.
\]
There are at most $m$ points in this intersection, so $q\leqslant m$.
The ranks are increasing positive integers, and therefore
$r_j\geqslant j$. Consequently,
\[
 0\leqslant S_m^*(e_A)
 =\sum_{j=1}^q\frac1{r_j}
 \leqslant\sum_{j=1}^q\frac1j
 \leqslant H_m,
\]
with the usual convention for an empty sum. Since $B\in\Gamma_0$,
we may take $A=B$, when equality holds. This proves
\eqref{eq:partial-sums}.

It follows that, for every $x\in E_0$,
\begin{align*}
 \sum_{j=1}^m|y_{b_j}^*(x)|
 &\geqslant |S_m^*(x)|\\
 &\geqslant S_m^*(e_B)-|S_m^*(x-e_B)|\\
 &\geqslant H_m\bigl(1-\|x-e_B\|\bigr).
\end{align*}
As $H_m\to\infty$, every $x$ in the open unit ball centered at $e_B$
lies outside $D_B(y^*)$. Hence the distance is at least one. It is
exactly one, the largest possible value, because $0\in D_B(y^*)$
and $\|e_B\|=1$.
\end{proof}

\begin{proof}[Proof of Theorem~\ref{thm:large}]
Since $|\Gamma_0|=\cfrak$, choose an injection
$\iota:\Gamma_0\to\Gamma$. Extend each coefficient function from
Proposition~\ref{prop:construction} by zero outside $\iota(\Gamma_0)$.
Extension by zero is a linear isometry
$\ell_\infty(\Gamma_0)\to\ell_\infty(\Gamma)$, so the resulting
sequence is normalized and weakly null. For each $B\in\Gamma_0$,
\eqref{eq:partial-sums} still holds with $e_{\iota(B)}$ in place of
$e_B$. The same estimate gives
$\dist(e_{\iota(B)},D_B(y^*))=1$ in $\ell_1(\Gamma)$.
No subsequence is strongly normal.
\end{proof}

\begin{remark}\label{rem:domains}
The failure of density does not mean that $D_B(y^*)$ is small.
For every infinite $A\subseteq\N\setminus\{1\}$, put
\[
 v_A=\tfrac12(e_A-e_{A\cup\{1\}}).
\]
In the construction of Proposition~\ref{prop:construction},
\[
 \sum_{n=1}^{\infty}|y_n^*(v_A)|
 =\frac12+\frac12\sum_{j=1}^{\infty}
                  \left(\frac1j-\frac1{j+1}\right)=1.
\]
The vectors $v_A$ have norm one and disjoint pairs of coordinates.
Their closed linear span is therefore isometric to $\ell_1(\cfrak)$.
The triangle inequality and absolute summation show that this entire
closed span lies in $D_{\N}(y^*)$, hence in every $D_B(y^*)$.
In particular, every such domain is nonseparable. Also,
$e_A\in D_B(y^*)$ whenever $A\cap B$ is finite.
\end{remark}

\begin{remark}\label{rem:context}
For any infinite $\Gamma$, choose a countably infinite
$\Gamma_1\subseteq\Gamma$. The norm-one coordinate projection onto
$\ell_1(\Gamma_1)$ has separable infinite-dimensional range. Thus
$\ell_1(\Gamma)$ has a separable quotient and its dual contains
strongly normal sequences by \cite[Theorem~3]{Sliwa}. Theorem~\ref{thm:large}
separates this existence statement from extraction out of every normal
sequence. There is also no conflict with the positive result for WCG
spaces: an uncountable $\ell_1(\Gamma)$ is not WCG. Indeed, by the
Schur property and the Eberlein--\v{S}mulian theorem its weakly compact
subsets are norm compact, so any weakly compactly generated subspace
is separable.
\end{remark}

\section{Cardinal bounds and the remaining questions}\label{sec:density}

For $f,g\in\N^{\N}$, write $f\leqslant^*g$ if $f(k)\leqslant g(k)$
for all sufficiently large $k$. The bounding number $\mathfrak b$ is
the least size of a family in $\N^{\N}$ with no common eventual upper
bound. We use $\dens E$ for the least cardinality of a norm-dense
subset of $E$.

Saxon and S\'anchez Ruiz \cite[Theorem~3]{SSR} proved that every
infinite-dimensional Banach space of density less than $\mathfrak b$
has an infinite-dimensional separable quotient. We recall the standard
selection argument in the proof of \cite[Corollary~11]{FKLS}; see also
\cite[proof of Theorem~2.1]{Brech}. This known argument already starts
with an arbitrary given normal sequence.

\begin{proposition}\label{prop:small}
If $\dens E<\mathfrak b$, then every normal sequence in $E^*$ has a
strongly normal subsequence.
\end{proposition}

\begin{proof}
Let $(y_n^*)$ be normal. By the definition of density, choose a
norm-dense set $D\subset E$ with $|D|<\mathfrak b$. For each $x\in D$,
choose $f_x\in\N^{\N}$ such that
\[
 n\geqslant f_x(k)\quad\Longrightarrow\quad
 |y_n^*(x)|\leqslant2^{-k}.
\]
There is $g\in\N^{\N}$ with $f_x\leqslant^*g$ for every $x\in D$.
Recursively choose a strictly increasing sequence $(n_k)$ with
$n_k\geqslant g(k)$. For each $x\in D$,
$|y_{n_k}^*(x)|\leqslant2^{-k}$ eventually, so
$\sum_k|y_{n_k}^*(x)|<\infty$. Thus the summability domain of
$(y_{n_k}^*)$ contains $D$. It is a linear subspace and is norm dense,
which is precisely strong normality.
\end{proof}

\begin{proof}[Proof of Theorem~\ref{thm:bounds}]
For infinite $\kappa$, $\dens\ell_1(\kappa)=\kappa$. Apply
Proposition~\ref{prop:small} when $\kappa<\mathfrak b$ and
Theorem~\ref{thm:large} when $\kappa\geqslant\cfrak$.
\end{proof}

\begin{corollary}\label{cor:independence}
CH implies $\neg\mathsf P_{\aleph_1}$, while
$\mathfrak b>\aleph_1$ implies $\mathsf P_{\aleph_1}$.
Consequently, assuming that ZFC is consistent, neither
$\mathsf P_{\aleph_1}$ nor its negation is provable in ZFC.
\end{corollary}

\begin{proof}
Under CH, $\aleph_1=\cfrak$, so apply the negative bound in
Theorem~\ref{thm:bounds}. Its positive bound applies when
$\mathfrak b>\aleph_1$. Martin's axiom implies
$\mathfrak b=\cfrak$; together with the relative consistency of CH
and of $\mathrm{MA}+\neg\mathrm{CH}$, this gives the independence
conclusion; see \cite[Section~6]{Blass} for the cardinal inequalities
under Martin's axiom.
\end{proof}

Corollary~\ref{cor:independence} does not decide
$\mathsf P_{\aleph_1}$ when $\mathfrak b=\aleph_1<\cfrak$.
The contrapositive of Proposition~\ref{prop:small} gives
$\neg\mathsf P_{\aleph_1}\Rightarrow\mathfrak b=\aleph_1$;
the converse is not established here.

\begin{problem}\label{prob:omegaone}
Does $\mathfrak b=\aleph_1$ imply $\neg\mathsf P_{\aleph_1}$?
Equivalently, is $\mathsf P_{\aleph_1}$ equivalent to
$\mathfrak b>\aleph_1$?
\end{problem}

More generally, define
\[
 \theta=\min\{\kappa\geqslant\aleph_0:\neg\mathsf P_\kappa\}.
\]
This cardinal exists by Theorem~\ref{thm:large}, and
$\mathfrak b\leqslant\theta\leqslant\cfrak$.
Failure of $\mathsf P_\kappa$ is preserved on increasing $\kappa$:
extend the functionals by zero, and observe that the summability domain
in the larger space is the product of the old domain with the
additional coordinate subspace.

\begin{problem}\label{prob:threshold}
Determine $\theta$. Must $\theta=\mathfrak b$?
\end{problem}

The following elementary observation explains a limitation of reducing
the index family in Proposition~\ref{prop:construction}. Write
$A\subseteq^*B$ when $A\setminus B$ is finite.

\begin{proposition}\label{prop:coinitial}
Suppose $\mathcal A\subseteq[\N]^\infty$ and every infinite
$B\subseteq\N$ almost contains some member of $\mathcal A$.
Then $|\mathcal A|=\cfrak$.
\end{proposition}

\begin{proof}
There is an almost disjoint family
$\{B_t:t\in2^{\N}\}\subseteq[\N]^\infty$: identify the finite
binary strings with natural numbers and let $B_t$ be the set of
initial segments of $t$. Choose $A_t\in\mathcal A$ with
$A_t\subseteq^*B_t$. If $s\neq t$ and $A_s=A_t$, this common
infinite set would be almost contained in the finite set
$B_s\cap B_t$, a contradiction. Hence $t\mapsto A_t$ is injective.
\end{proof}

Such a family would suffice for the same harmonic obstruction. Indeed,
if $A\subseteq^*B$, enumerate $A\cap B=\{c_1<c_2<\cdots\}$ and set
$d=|A\setminus B|$. The rank of $c_j$ in $A$ is at most $j+d$.
Consequently,
\[
 \left\|\sum_{j=1}^m y_{c_j}^*\right\|\leqslant H_m,
 \qquad
 \sum_{j=1}^m y_{c_j}^*(e_A)\geqslant H_{m+d}-H_d,
\]
where $H_0=0$. The ratio $(H_{m+d}-H_d)/H_m$ tends to one, giving
the same exclusion of the open unit ball centered at $e_A$ from the
summability domain along $B$. Proposition~\ref{prop:coinitial} shows
that this sufficient condition cannot produce a smaller index family.
It does not rule out other constructions at smaller densities.

\section{A further Banach sequence space}\label{sec:sequence-space}

The failure of extraction is not confined to spaces of the form
$\ell_1(\Gamma)$. Consider
\[
 X=\left\{x\in\K^{\N}:\|x\|_X=
 \sup_{\substack{F\subseteq\N\text{ finite}\\F\neq\varnothing}}
 \frac1{\sqrt{|F|}}\sum_{n\in F}|x_n|<\infty\right\}.
\]
This is a Banach norm for the usual weak-$\ell_2$ sequence space.
Completeness follows directly: a norm-Cauchy sequence has a limit in
each coordinate, and passing to the limit in every finite-set estimate
proves convergence in $\|\cdot\|_X$.

\begin{proposition}\label{prop:marcinkiewicz}
The coordinate functionals $(\delta_n)$ form a normalized weakly null
sequence in $X^*$ with no strongly normal subsequence. The space $X$
is not isomorphic to any $\ell_1(\Gamma)$.
\end{proposition}

\begin{proof}
We have $\|\delta_n\|=1$, with equality attained at the $n$th unit
vector. For distinct $n_1,\ldots,n_m$ and scalars $|\varepsilon_j|=1$,
\[
 \left\|\sum_{j=1}^m\varepsilon_j\delta_{n_j}\right\|\leqslant\sqrt m.
\]
If some $F\in X^{**}$ satisfied $|F(\delta_{n_j})|\geqslant\epsilon>0$
on an infinite subsequence, choose the scalars to make
$\varepsilon_jF(\delta_{n_j})$ nonnegative real. Then
$m\epsilon\leqslant\|F\|\sqrt m$ for every $m$, a contradiction.
Thus $(\delta_n)$ is weakly null.

Fix $B=\{b_1<b_2<\cdots\}$ and put
$a_j=\sqrt j-\sqrt{j-1}$. Let $x_B(b_j)=a_j$, with all other
coordinates zero. The numbers $a_j$ decrease and sum to $\sqrt m$
over the first $m$ indices, so $\|x_B\|_X=1$. The functionals
$v_m^*=m^{-1/2}\sum_{j=1}^m\delta_{b_j}$ have norm one and satisfy
$v_m^*(x_B)=1$. If $x\in D_B(\delta)$, then
\[
 |v_m^*(x)|\leqslant m^{-1/2}\sum_{j=1}^{\infty}|x_{b_j}|
 \longrightarrow0.
\]
It follows that $\dist(x_B,D_B(\delta))=1$.

Finally, the natural inclusion $\ell_2\to X$ is bounded by the
Cauchy--Schwarz inequality. Its unit vectors give a normalized weakly
null sequence in $X$, so $X$ fails the Schur property. Every
$\ell_1(\Gamma)$ has the Schur property, completing the proof.
\end{proof}

\end{document}